\documentclass[12pt]{amsart}
\usepackage{amsmath, amssymb, amscd, a4wide, mathrsfs}

\numberwithin{equation}{section}

\theoremstyle{plain}
\newtheorem{thm}{Theorem}[section]
\newtheorem{lem}[thm]{Lemma}
\newtheorem{prop}[thm]{Proposition}

\newcommand{\thmref}[1]{Theorem~\ref{#1}}
\newcommand{\lemref}[1]{Lemma~\ref{#1}}

\theoremstyle{definition}

\newtheorem*{hypo*}{Hypothesis~I}

\newcommand{\hyporef}[1]{Hypothesis~I}

\newcommand{\mc}{\mathcal}

\begin{document}

\title[On the signs of Fourier coefficients of Hilert cusp forms ]{On the signs of Fourier coefficients of Hilbert cusp forms}

\author{Ritwik Pal}
\address{Department of Mathematics\\
	Indian Institute of Science\\
	Bangalore - 560012, India.}
\email{ritwik.1729@gmail.com, ritwikpal@iisc.ac.in}

\date{}
\subjclass[2010]{11F30, 11F41} 
\keywords{ Hilbert modular forms, first sign change, Fourier coefficients, distribution of eigenvalues}

\begin{abstract}
We prove that given any $\epsilon > 0$ and a primitive adelic Hilbert cusp form $f$ of weight $k=(k_1,k_2,...,k_n) \in (2 \mathbb{Z})^n$ and full level, there exists an integral ideal $\mathfrak{m}$ with $N(\mathfrak{m}) \ll_{\epsilon} Q_{f}^{9/20+ \epsilon} $ such that the $\mathfrak{m}$-th Fourier coefficient of $C_{f} (\mathfrak{m})$ of $f$ is negative. Here $n$ is the degree of the associated number field, $N(\mathfrak{m})$ is the norm of integral ideal $\mathfrak{m}$ and $Q_{f}$ is the analytic conductor of $f$. In the case of arbitrary weights, we show that there is an integral ideal $\mathfrak{m}$ with $N(\mathfrak{m}) \ll_{\epsilon} Q_{f}^{1/2 + \epsilon}$ such that $C_{f}(\mathfrak{m}) <0$. We also prove that when $k=(k_1,k_2,...,k_n) \in (2 \mathbb{Z})^n$, asymptotically half of the Fourier coefficients are positive while half are negative.  
\end{abstract}
\maketitle 

\section{Introduction} 
 This paper is concerned with a quantitative result on the study of signs of Fourier coefficients of Hilbert cusp forms. This theme of research has seen a lot of activity in the recent past; here we just recall the landmark result of K. Matom\"aki for elliptic Hecke cusp forms, which also sets the ground of our results to follow. Let $Q_{f} \asymp k^2 N$ be the analytic conductor of an elliptic newform $f$ of level $N$ and weight $k$. Then it was proved in \cite{km1} that the first negative eigenvalue of $f$ occurs at some $n_0 \geq 1$ with $n_0 \ll Q_{f}^{3/8}$. The method in \cite{km1} is based on further refinement of that in \cite{sound} wherein a variety of results on statistical distribution of signs of Fourier coefficients of newforms were studied.
 
In the case of Hilbert newforms of arbitrary weight and level, the only known result seems to be the work of Meher and Tanabe (see \cite[Theorem 1.1, 1.2]{jm}). To describe their result let us introduce the following notation. Let $F$ be a totally real number field of degree $n$ associated with the adelic Hilbert newform $f$. Let $\{ C(\mathfrak{m}) \}_{\mathfrak{m}}$ denote the Fourier coefficients of $f$, indexed by the integral ideals $\mathfrak{m}$ and $Q_{f}$ denote the analytic conductor of $f$ (see next section for the definition). Then in \cite{jm} it is shown that the the sequence $\{C(\mathfrak{m})\}_{m}$ changes sign infinitely often and more quantitatively the main result of \cite{jm} states that there exists an integral ideal $\mathfrak{m}$ with
\[N(\mathfrak{m}) \ll_{n, \epsilon} Q_{f}^{1+\epsilon}\]
such that $C(\mathfrak{m}) <0$, where $N(\mathfrak{m})$ is the norm of integral ideal $\mathfrak{m}$. One of the aim of this paper is to improve upon this result. Our main result is the following.   
\begin{thm} \label{first sin chng}
Let $f$ be a Hilbert newform of weight $k=(k_1,k_2,...,k_n)$ and full level. Let $C(\mathfrak{m})$ denote the Fourier coefficient of $f$ at the ideal $\mathfrak{m}$. Then for any arbitrary $\epsilon >0$,
\begin{enumerate}
\item[{(i)}]
when $k_1,k_2,...,k_n$ are all even, we have $C(\mathfrak{m})<0$ for some ideal $\mathfrak{m}$ with 
\[ N(\mathfrak{m}) \ll_{n, \epsilon} Q_{f}^{\frac{9}{20}+ \epsilon}; \]
\item[{(ii)}]
otherwise we have $C(\mathfrak{m})<0$ for some ideal $\mathfrak{m}$ with
\[ N(\mathfrak{m}) \ll_{n, \epsilon} Q_{f}^{\frac{1}{2}+ \epsilon}. \]
\end{enumerate}
\end{thm}
The bound in the case (i) is stronger as the Ramanujan conjecture is known in this case but it seems not yet in case (ii). 

The main difference in our approach with that in \cite{jm} is that instead of using the bound $|\widetilde{C}(\mathfrak{p}) | + |\widetilde{C}(\mathfrak{p}^2)| \geq 1/2$ (see \cite[Prop 4.5]{jm}) for certain `good' primes $\mathfrak{p}$, we work directly with the Hecke relation $\widetilde{C}(\mathfrak{p})^2 -\widetilde{C}(\mathfrak{p}^2)=1$ (see \cite[2.23]{shimu}), where
\[\widetilde{C}(\mathfrak{m}):= \frac{C(\mathfrak{m})}{N(\mathfrak{m})^{\frac{k_0-1}{2}}} \quad  \text{and} \quad k_0:= \max \{k_1,k_2,...,k_n\}.\]
 Like other results available in the topic, we consider upper and lower bounds of a suitable weighted partial sum of normalized Fourier coefficients $\widetilde{C}(\mathfrak{m})$:
\begin{equation} \label{S(f,x)}
 S(f,x) := \sum_{N(\mathfrak{m})\leq x} \widetilde{C}(\mathfrak{m}) \log(\frac{x}{N(\mathfrak{m})}).
\end{equation}
 Using the convexity principle for automorphic $L$-functions and Perron's formula (see eg. \cite[chapter 5]{iwaniec}) we get an upper bound of $S(f,x)$ in terms of $Q_{f}$ and $x$. For the lower bound of $S(f,x)$ in the first case of \thmref{first sin chng}, we adopt a method similar in the spirit of \cite[Theorem 1]{sound}. The introduction of the weighted $\log$ in \eqref{S(f,x)} is necessary to deal with convergence issues while working with the Perron formula, unlike the case in \cite{sound}. To find a lower bound, we work with 
\[ T(f,x):= \sum_{N(\mathfrak{m}) \leq x} \widetilde{C}(\mathfrak{m})\]
and recover a corresponding lower bound of $S(f,x)$ by partial summation.

In the second result of this paper, we extend Theorem 1.1 of \cite{km} to the case of primitive Hilbert cusp forms. The method used here relies upon that in \cite[Theorem 1.1]{km}.
\begin{thm} \label{ratio}
Let $k_1, k_2,...,k_n$ be all even and $f$ be an adelic Hilbert newform of weight $k=(k_1,k_2,...,k_n)$ and full level. Then one has,
\[\lim_{x \rightarrow \infty} \frac{|\{\mathfrak{m} | N(\mathfrak{m})\leq x, C(\mathfrak{m})>0\}|}{|\{\mathfrak{m} | N(\mathfrak{m}) \leq x\}|} =\lim_{x \rightarrow \infty} \frac{|\{\mathfrak{m} | N(\mathfrak{m})\leq x, C(\mathfrak{m})<0\}|}{|\{\mathfrak{m} | N(\mathfrak{m}) \leq x\}|} =\frac{1}{2}. \]
\end{thm}
The restriction of the weight in \thmref{ratio} is due to the use of Sato-Tate theorem \cite[pp 1-6]{sato}, which is available in the case when all of $k_1, k_2, ..., k_n$ are even. We also make a note of the fact that our proofs of both \thmref{first sin chng} and \thmref{ratio} can be generalized to an arbitrary fixed level $\mathfrak{n}$ and would get same result. For simplicity we restrict ourselves to the case of full level.      

\textbf{Acknowledgement.} The author thanks Prof. Soumya Das for many valueable discussions and suggestions. He also thanks NBHM for financial support and IISc, Bengaluru, where this work was done.
\section{Notation and preliminaries}\label{prelim}
The setting of the paper is as follows (see \cite{shimu} for detailed discussion): let $F$ be a totally real number field of degree $n$ and $\mathcal{O}_{F}$ be the ring of integers of $F$. Throughout this paper, the integral ideals and prime ideals of $F$ will be denoted by gothic symbols like $\mathfrak{m}$ and $\mathfrak{p}$ respectively. We will   denote the set of all integral ideals and prime ideals of $F$ by $\mathcal{I}$ and $\mathcal{P}$ respectively. Let $h$ be the narrow class number of $F$ and $\{c_{\nu}:= t_{\nu } \mathcal{O}_{F}\}_{1}^{h}$ be the complete set of representatives of the narrow class group, where $t_{\nu}$ being finite ideles. Let $\mathcal{D}_{F}$ be the different ideal of $F$. For each $c_{\nu}$ consider the following subgroup of $GL_{2}(F)$:
\begin{align*}
\Gamma(c_{\nu} \mathcal{D}_{F}, \mathcal{O}_{F})=\Big\{\left(\begin{array}{cc}
   a & b\\
   c & d
\end{array}\right) \mid a,d \in \mathcal{O}_{F}, c\in c_{\nu} \mathcal{D}_{F}, b\in c_{\nu}^{-1} \mathcal{D}_{F}^{-1},  ad-bc \in \mathcal{O}_{F}^{*}\Big\}.
\end{align*}
A classical Hilbert cusp form $f_{\nu}$ of weight $(k_1,k_2,..., k_n)$ on $\Gamma(c_{\nu} \mathcal{D}_{F}, \mathcal{O}_{F})$ has the following Fourier expansion:
\[f_{\nu} (z)= \sum_{0 \ll \eta \in c_{\nu}} a_{\nu}(\eta) \exp(Tr(\eta z)).\]
To talk about newforms, following \cite{shimu}, we associate an $h$-tuple $(f_{1},f_{2},...,f_{h})$ to an adelic Hilbert cusp form $f$. Recall that $f$ is associated with an automorphic form on $GL_{2}(\mathbb{A}_{F})$, where $\mathbb{A}_{F}$ is the adele ring of $F$ and the Fourier coefficients Fourier coefficients $\{C(\mathfrak{m}) \}_{\mathfrak{m}}$ of $f$ are given by the following relation (see \cite[equation: 2.17, 2.22]{shimu})
\[C(\mathfrak{m}) =a_{\nu}(\eta) \eta^{-k_0/2} N(\mathfrak{m})^{k_0/2},\] 
where $\mathfrak{m}=\eta c_{\nu}^{-1}$ for a unique $\nu$ and some totally positive element $\eta$ in $F$. The Hecke-theory for adelic Hilbert newforms (sometimes referred to as primitive forms) was established by Shimura in \cite{shimu}. When $f$ is an adelic Hilbert newform, usually one has the normalization $C(\mathcal{O}_{F}) = 1$ (see \cite[pp 650]{shimu}), which we would assume throughout this paper.
\subsection{General notation}
Suppose $ h(x) >0$ is defined on a subset $ B$ of $\mathbb{R}$ and let $g(x)$ be such that, $g(x): B \mapsto \mathbb{R}$. Whenever we write $g(x) \ll h(x)$ or $g(x) = O(h(x))$ or $g(x) \ll_\epsilon h(x)$, it will always mean that
\[ |g(x)| \leq M \cdot h(x), \quad \text{for all  } x \in B \text{  and for some  } M >0 ,\]
in the last case $M$ may depend on $\epsilon$.
Let $u(x), v(x) : \mc B \mapsto \mathbb{R}$, where $\mc B $ is a subset of $\mathbb{R}$. The notations $u(x)= o(v(x))$ or $u(x) \sim v(x)$ respectively mean that 
\[ \underset{x \rightarrow \infty} \lim (u(x)/v(x)) =0 \quad \text{or} \quad \underset{x \rightarrow \infty} \lim (u(x)/v(x)) =1 . \]
Whenever we use $\sum^{\#}$, it will always signify that the summation is restricted to square-free integral ideals.
\subsection{L- function of an adelic Hilbert newform}
Let $f$ be an adelic Hilbert newform of weight $(k_1,k_2,..., k_n)$ and full level. Let $\{C(\mathfrak{m})\}_{\mathfrak{m}}$ denote its Fourier coefficients. Then the normalized $L$-function attached to $f$ is an absolutely convergent Dirichlet series for $\Re s >1$ (see \cite[section 4.4]{raghuram} for the details of this subsection), given by
\[L(s,f):= \underset{\mathfrak{m} \in \mathcal{I}} \sum \frac{\widetilde{C}(\mathfrak{m})}{N(\mathfrak{m})^s}\]
where $\mathcal{I}$ and $\widetilde{C}(\mathfrak{m})$ are same as defined before. It is well known that $L(s,f)$ admits an Euler product:
\[L(s,f)= \underset{\mathfrak{p} \in \mathcal{P}}\prod (1- \widetilde{C}(\mathfrak{p}) N(\mathfrak{p})^{-s}+  N(\mathfrak{p})^{-2s} )^{-1} \]
and it can be analytically continued to the whole complex plane $\mathbb{C}$. Let us put
\[L_{\infty}(s,f):= N(\mathcal{D}_{F}^2)^{\frac{s}{2}} \prod_{j=1}^{n} (2 \pi)^{-(s+\frac{k_{j}-1}{2})} \Gamma(s+\frac{k_{j} -1}{2}). \]
The completed $L-$function $\Lambda$ is  then defined by
\[ \Lambda(s,f) = L(s,f) L_{\infty}(s,f) ,\]
which satisfies the functional equation
\[\Lambda(s,f)= i^{\sum_{j}k_j} \Lambda(1 -s, f).\]
\subsection{Analytic conductor}
With the given data in the previous subsection we define the analytic conductor at $s= 1/2$ (here we follow \cite[chapter 5]{iwaniec}, where these objects are defined for a more general $L$-function) $Q_{f}$ of $L(f,s)$ (or $f$ for brevity) as 
\[ Q_{f} := N(\mathcal{D}_{F}^2) \prod_{j=1}^{n} (\frac{k_j+5}{2})(\frac{k_j+7}{2}) \asymp (\prod_{j=1}^{n} k_{j})^2 N(\mathcal{D}_{F}^2).\]
\section{Proof of Theorem 1.1}
Let $f$ be as in \thmref{first sin chng}. Let $y>0$ be such that for all $\mathfrak{m} \in \mathcal{I}$ with $N(\mathfrak{m}) \leq y$ we have $C(\mathfrak{m}) \geq 0 $. We will estimate $y$ in terms of $Q_{f}$ with an implied absolute constant (hence we may assume $y$ to be bigger than some absolute constant at some appropriate place) from the comparison of upper and lower bounds of the sum (for a suitable $x= y^{\theta}$ for some $\theta$ to be specified later)
\[ S(f,x) := \sum_{N(\mathfrak{m})\leq x} \widetilde{C}(\mathfrak{m}) \log(\frac{x}{N(\mathfrak{m})}).\]
\subsection{Upper bound}
First let us recall a convexity bound result (see Lemma 3.7, \cite{qu}):
let $\epsilon >0$ be arbitrary and $ 0 < \sigma < 1$, where $s= \sigma+it$. Then we have
\begin{equation} \label{harcos}
 L(\sigma +it, f) \ll_{\epsilon} \big ( (1+ | t |)^{2n+1} Q_{f} \big )^{\frac{1-\sigma}{2} + \epsilon}. 
\end{equation}
From Perron's formula we get
\[ S(f,x) = \sum_{N(\mathfrak{m}) \leq x} \widetilde{C}(\mathfrak{m})\log(\frac{x}{N(\mathfrak{m})}) = \frac{1}{2 \pi i} \int_{(2)} L(s,f) \frac{x^{s}}{s^2} ds. \]
Here $(2)$ in the limit of the integral means the the contour of the integral is $\Re s= 2$. Using \eqref{harcos} it is clear that the integral is absolutely convergent for $\Re s \geq \frac{2n}{2n+1}$. So we shift the line of integration to $\sigma = \frac{2n}{2n+1} $ (horizontal integrals do not contribute owing to \eqref{harcos}). Further using \eqref{harcos} we obtain the estimate
\begin{equation} \label{ub}
S(f,x) \ll_{\epsilon} Q_{f}^{\frac{1}{2(2n+1)}+ \epsilon} x^{\frac{2n}{2n+1}}.
\end{equation}

\subsection{Lower bound}
Let us recall the function $T(f,x)$ defined in the introduction:
\begin{equation} \label{T}
T(f,x):= \sum_{N(\mathfrak{m}) \leq x} \widetilde{C}(\mathfrak{m}) .
\end{equation}
In this subsection we will find a lower bound of $T(f,x)$ for some suitable $x$. At this point let us recall the following result about the coefficients $\widetilde{C}(\mathfrak{m})$ (see \cite[equation 2.23]{shimu}): for any unramified prime ideal $\mathfrak{p}$ (i.e. $\mathfrak{p} \nmid \mathcal{D}_{F}$), one has the Hecke relation
\begin{equation} \label{C(p)}
\widetilde{C}(\mathfrak{p})^2 = 1+ \widetilde{C}(\mathfrak{p}^{2}).
\end{equation}
Hence for all unramified primes $\mathfrak{p}$ satisfying $N(\mathfrak{p})\leq y^{\frac{1}{2}}$, we have $\widetilde{C}(\mathfrak{p}) \geq 1$. Following \cite{sound} let us introduce an auxiliary multiplicative function $h\equiv h_y : \mathcal{I} \rightarrow \mathbb{R}$ defined by,
\begin{equation}
h_y(\mathfrak{p})=
\begin{cases}
1 \quad\quad \text{if  } N(\mathfrak{p})\leq y^{\frac{1}{2}} ;\\
0 \quad \quad\text{if  } y^{\frac{1}{2}} < N(\mathfrak{p})\leq y ;\\
- 2 \quad \text{if  } N(\mathfrak{p}) > y
\end{cases} 
\end{equation}
and $h_y(\mathfrak{p}^v)=0$ for $v\geq 2$. Recall that for adelic Hilbert newforms of weight $(k_1, k_2,..., k_n) \in (2 \mathbb{Z})^{n}$, by the Ramanujan-Petersson bound one has $|\widetilde{C}(\mathfrak{p})| \leq 2$  (see \cite[Theorem 1]{bla}). We will now prove the following lemma for further use (see \cite[Lemma 2.1]{sound} for the case when $F= \mathbb{Q}$).
\begin{lem} \label{sum}
For any $\epsilon >0$, we have 
\begin{equation} \label{h}
\sum_{N(\mathfrak{m}) \leq y^u } h_y(\mathfrak{m})= \frac{c_F}{\zeta_F(2)} y^u (\rho(2u)- 2 \log u) \{ 1+ O(\frac{1}{\log y}) \}
\end{equation}
uniformly for $1 \leq u \leq \frac{3}{2}$, where $\zeta_F$ is the Dedekind zeta function, $c_F$ is the residue $\zeta_F$ at the pole $s=1$ and $\rho (u)$ is the Dickman function, defined as the unique continuous solution of the difference-differential equation
\[u \rho '(u) +\rho (u-1) =0 \quad (u>1), \quad \rho(u)=1 \quad (0< u \leq 1).\] 
\end{lem}
\begin{proof}
For $\mathfrak{m} \in \mathcal{I}$ define
\[P(\mathfrak{m}) := \max \{N(\mathfrak{p}) | \mathfrak{p} \in \mathcal{P}, \mathfrak{p}| \mathfrak{m} \}.\]
We also define 
\[ \psi(x,y):= | \{\mathfrak{m} : N(\mathfrak{m}) \leq x , P(\mathfrak{m}) \leq y \}| . \]
 For $1 \leq u \leq 3/2$, from the definition of $h_{y}$ we get,
\begin{equation} \label{hnew}
\sum_{N(\mathfrak{m})\leq y^{u}} h_{y}(\mathfrak{m}) = \psi^{\#}(y^u, y^{\frac{1}{2}}) - 2 \sum_{y\leq N(\mathfrak{p}) \leq y^u} \underset{N(\mathfrak{l}) \leq \frac{y^u}{N(\mathfrak{p})}}{\sideset{}{^\#} \sum} 1 .
\end{equation}
Here the sign $\#$ signifies that the sum is taken over square-free integral ideals. Let us first estimate the second term of \eqref{hnew}. We state the following result whose proof is given after this lemma. 
\begin{equation} \label{sumnorm}
\underset{N(\mathfrak{l}) \leq x}{\sideset{}{^\#}\sum} 1 = \frac{c_F}{\zeta_{F} (2)} x + O(\frac{x}{\log x}),
\end{equation}
where $c_F$ is the residue of the Dedekind zeta function $\zeta_F$ at the pole $s=1$. For the second term of the right hand side of \eqref{hnew}, arguing exactly as in \cite[pp 8-9]{sound} gives an upper bound
\begin{equation} \label{second term}
\sum_{y\leq N(\mathfrak{p}) \leq y^u} {\underset{N(\mathfrak{l}) \leq \frac{y^u}{N(\mathfrak{p})}}{\sideset{}{^\#} \sum}} 1 = \frac{c_F}{\zeta_{F}(2)}y^{u} \log u + O(\frac{y^u}{\log y}).
\end{equation}
For \eqref{second term} we have also used the fact (see \cite[Prop 2]{lebacque}) that,
\begin{equation} \label{sum N(p)}
\sum_{N(\mathfrak{p}) \leq x} \frac{1}{N(\mathfrak{p})}= \log \log x +B + o(1),
\end{equation}
where $B$ is a constant depending only on $F$.

The first term in \eqref{hnew} is given by the following (the proof is given after this lemma):
\begin{equation} \label{original first term}
\psi^{\#}(y^{u},y^{\frac{1}{2}}) = \frac{c_{F}}{\zeta_{F}(2)} \cdot \rho(2u) y^{u} + O(\frac{y^{u}}{\log y}).
\end{equation}
Thus the \lemref{sum} follows from \eqref{second term} and \eqref{original first term}.  
\end{proof}
The proofs of the two results that we used in \lemref{sum} are verbatim with the case of integers (i.e. $\mathcal{O}_{F}= \mathbb {Z}$). However for convenience we give the proofs here.
\begin{prop}
With the notations as given above, one has
\[\underset{N(\mathfrak{l}) \leq x}{\sideset{}{^\#}\sum} 1 = \frac{c_F}{\zeta_{F} (2)} x + O(\frac{x}{\log x}).\]
\end{prop}
\begin{proof}
The M\"obius function $\mu: \mathcal{I} \mapsto \{1,0,-1 \}$ is defined by $\mu(\mathfrak{p})=-1$, $\mu(\mathfrak{p}^v)=0$ for $v \geq 2$ on $\mathcal{P}$ and extended multiplicatively to $\mathcal{I}$. It satisfies the usual properties
\begin{equation} \label{mu}
 \mu^{2} (\mathfrak{m}) = \underset{\mathfrak{l}^2 | \mathfrak{m}} \sum \mu (\mathfrak{l}) \quad \text{and} \quad \underset{\mathfrak{m} \in I} \sum \frac{\mu(\mathfrak{m})}{N(\mathfrak{m})^2} =\frac{1}{\zeta_{F}(2)}, 
\end{equation}
where $\zeta_{F}$ is the Dedekind zeta function. From \cite[Theorem 11.1.5]{mur} we write (please see the remark after this proof)
\[\underset{N(\mathfrak{l}) \leq x}{\sum} 1 = \frac{c_F}{\zeta_{F} (2)} x + O(\frac{x}{\log x}).\]
Now we calculate

\begin{align*}
 \underset{N(\mathfrak{l}) \leq x}{\sideset{}{^\#}\sum} 1= \underset{N(\mathfrak{l}) \leq x}{\sum} \mu^{2}(\mathfrak{l})= &\underset{N(\mathfrak{l}) \leq x}{\sum} \underset{\mathfrak{t}^2 | \mathfrak{l}}{\sum} \mu(\mathfrak{t})  \\
&= \underset{N(\mathfrak{t}^2) \leq x}{\sum} \mu(\mathfrak{t}) \big(\underset{N(\mathfrak{s}) \leq \frac{x}{N(\mathfrak{t})^2}}{\sum} 1 \big) \\
&= \underset{N(\mathfrak{t}^2) \leq x}{\sum} \mu(\mathfrak{t}) \big(c_{F} \frac{x}{N(\mathfrak{t}^2)}+ O(\frac{\frac{x}{N(\mathfrak{t}^2)}}{\log(\frac{x}{N(\mathfrak{t}^2)})})\big) \\
&= c_{F} x \cdot \underset{N(\mathfrak{t}^2) \leq x}{\sum} \frac{\mu({\mathfrak{t})}}{N(\mathfrak{t}^2)} +\underset{N(\mathfrak{t}^2) \leq x}{\sum} \frac{\mu({\mathfrak{t})}}{N(\mathfrak{t}^2)} O(\frac{x}{\log x}).  \\
&= \frac{c_F}{\zeta_{F} (2)} x + O(\frac{x}{\log x}).
\end{align*}
For the last equation we have used the fact that
\[\underset{N(\mathfrak{t}^2) \leq x}{\sum} \frac{\mu({\mathfrak{t})}}{N(\mathfrak{t}^2)}= \frac{1}{\zeta_{F}(2)}+O(x^{\frac{-1}{2}}).\]
\end{proof}
\textbf{Remark:} The error term above is actually of the magnitude $x^{1-\frac{1}{n}}$ (for eg. see \cite[Theorem 11.1.5]{mur}). However this error term bound will not benefit us because of the error term in the next result is of the magnitude $\frac{x}{\log x}$. 

\begin{prop}
With the notations as given above, one has
\[\psi^{\#}(y^{u},y^{\frac{1}{2}}) = \frac{c_{F}}{\zeta_{F}(2)} \cdot \rho(2u) y^{u} + O(\frac{y^{u}}{\log y}).
\]
\end{prop}
\begin{proof}
From \cite[Lemma 4.1]{moore}, for $0 <\delta<1$ one has 
\begin{equation} \label{first term}
\psi(x,x^{\delta}) = c_{F} \cdot \rho(\frac{1}{\delta}) x + O(\frac{x}{\log x}).
\end{equation}
Now arguing as in the previous proposition we get
\begin{align*}
\psi(y^{u},y^{\frac{1}{2}})^{\#}= \underset{N(\mathfrak{l}) \leq y^u, P(\mathfrak{l}) \leq y^{\frac{1}{2}}}{\sum} \mu^{2}(\mathfrak{l})= \underset{N(\mathfrak{t}^2) \leq y^u, P(\mathfrak{t}) \leq y^{\frac{1}{2}}}{\sum} \mu(\mathfrak{t}) \psi(\frac{y^u}{N(\mathfrak{t}^2)},y^{\frac{1}{2}}).
\end{align*}
Now we put the estimate of \eqref{first term} in the right hand side of the last equation and the rest of the proof would follow exactly same as in \cite[pp 190-191]{ivic} where he has proved the same result in the case of integers (i.e. $\mathcal{O}_{F}= \mathbb{Z}$). From here \cite{ivic} used certain properties of the Dickman function $\rho$ and $\mu$ function, which are same as in this case (includng the $\mu$ function).  
\end{proof}

It is clear from \lemref{sum} that whenever $\rho(2u)- 2 \log u > 0$ and $y$ is large enough, one has $\sum_{N(\mathfrak{m}) \leq y^u } h_y(\mathfrak{m}) > 0$. It is known from \cite{sound} that $\rho(2u)- 2 \log u > 0$ for all $u< \kappa$, where $\kappa$ is the solution of the equation $\rho(2u)= 2 \log u$ and $\kappa> \frac{10}{9}$.   
\begin{lem} \label{th}
Let $f$ be a primitive form of weight $(k_1,k_2,...k_n) \in (2 \mathbb{Z})^n$. Then for any fixed $u$ with $1 \leq u < \kappa$, we have 
\[ T(f,y^{u}) \geq  \sum_{N(\mathfrak{m}) \leq y^{u}} h_{y}(\mathfrak{m}) \gg_{u} y^u. \]
\end{lem}
\begin{proof}
Let $g_{y}: \mathcal{I} \rightarrow \mathbb{R}$ be the multiplicative function defined by the Dirichlet convolution identity
\[ \widetilde{C}(\mathfrak{m})= (g_{y} \ast h_{y}) (\mathfrak{m}). \]
Hence for any prime ideal $\mathfrak{p}$, from the definition of $h_y$ we have
\[ g_{y}(\mathfrak{p})= \widetilde{C}(\mathfrak{p})- h_{y}(\mathfrak{p})\geq 0. \]
By multiplicativity $g_{y}(\mathfrak{m}) \geq 0$ for any square-free integral ideal $\mathfrak{m}$. Hence for $u < \kappa$ (so that by the discussion after \lemref{sum}, one has $\sum_{N(\mathfrak{m})\leq y^{u}} h_{y}(\mathfrak{m}) > 0$), we have
\[T(f,y^{u}) \geq \underset{N(\mathfrak{m}) \leq y^{u}}{\sideset{}{^\#} \sum} \widetilde{C}(\mathfrak{m})= \underset{N(\mathfrak{d})\leq y^{u}}{\sideset{}{^\#}\sum} g_{y}(\mathfrak{d}) \sum_{N(\mathfrak{l})\leq \frac{y^u}{N(\mathfrak{d})}} h_{y}(\mathfrak{l}) \geq \sum_{N(\mathfrak{l})\leq y^{u}}h_{y}(\mathfrak{l}). \]
The last inequality follows from the fact that every term on the left hand side of the inequality is positive and $g_y (\mathcal{O}_{F})= 1.$ 
\end{proof}

\begin{lem} \label{sh}
Let $f$ be a primitive form of weight $(k_1,k_2,...k_n) \in (2 \mathbb{Z})^n$. Then for any fixed $u$, with $1 \leq u < \kappa$, one has
\[S(f,y^u) \gg_{u} y^u .\]
\end{lem}
\begin{proof}
We have,
\begin{equation*}
\begin{split}
S(f,y^u)
=& \underset{N(\mathfrak{m}) \leq y^u} \sum \widetilde{C}(\mathfrak{m}) \log y^u - \underset{N(\mathfrak{m}) \leq y^u} \sum \widetilde{C}(\mathfrak{m}) \log N(\mathfrak{m}) \\
=& T(f,y^u) \log y^u- \sum_{n \leq y^u} a(n) \log n, 
\end{split}
\end{equation*}
where $a(n):= \sum_{N(\mathfrak{m})=n} \widetilde{C}(\mathfrak{m})$. Now using  the Abel-summation formula one has,
\[T(f,y^u) \log y^u- \sum_{n \leq y^u} a(n) \log n = \int_{1}^{y^u} \frac{T(f,t)}{t} dt.\]
 Now using \lemref{th} and the fact that $T(f,t)\geq 0$ for $1 \leq t \leq y$, the lemma follows immediately.  
\end{proof}
\subsection{Proof of Theorem 1.1}
\begin{enumerate}
\item[{(i)}]
When $(k_1,k_2,...,k_n) \in (2 \mathbb{Z})^n $, putting $u=\frac{10}{9}$, i.e. $x= y^{\frac{10}{9}}$ and using the comparison of the upper and the lower bounds of $S(f,y^\frac{10}{9})$ from \eqref{ub} and \lemref{sh}  we have
\[ y \ll_{n, \epsilon} Q_{f}^{\frac{9}{20}+ \epsilon}. \]
\item[{(ii)}]
For any other weight, note that for $x \leq y$, one has
\[S(f,x)= \sum_{N(\mathfrak{m}) \leq x} \widetilde{C}(\mathfrak{m})\log(\frac{x}{N(\mathfrak{m})}) \gg \sum_{N(\mathfrak{m}) \leq \frac{x}{2}} \widetilde{C}(\mathfrak{m}).\]
We put $x=y$. Now for $N(\mathfrak{p}) \leq y^{1/2}$, one has $\widetilde{C}(\mathfrak{p}) \geq 1$. It implies that
\[\underset{N(\mathfrak{m}) \leq \frac{y}{2}} \sum \widetilde{C}(\mathfrak{m}) \gg \underset{N(\mathfrak{m}) \leq {\frac{y}{2}}}{\sideset{}{^\#}\sum} \widetilde{C}(\mathfrak{m})= \psi^{\#} (\frac{y}{2}, (\frac{y}{2})^{\frac{1}{2}}) \gg y,\]
where the last inequality follows from \eqref{original first term}. Now comparing the upper and the lower bounds of $S(f,y)$, we get 
\[ y \ll_{n, \epsilon} Q_{f}^{\frac{1}{2}+ \epsilon}. \]
\end{enumerate}
 
 \section{Proof of Theorem 1.2}
 The idea of the proof of \thmref{ratio} is based on the work of K. Matom\"aki and M. Radziwi\l{}\l{}, who proved the result in the case of elliptic newforms of full level and weight (see \cite{km}). The following lemma will be useful to prove \thmref{ratio}.
 
 \begin{lem} \label{3.1}
 	Let $K,L : \mathbb{R}^{+} \rightarrow \mathbb{R}^{+} $ be such that $K(x) \rightarrow 0$ and $L(x) \rightarrow \infty$ for $x \rightarrow \infty$. Let $g : \mathcal{I} \rightarrow \{-1,0,1\} $ be a multiplicative function such that for every $x \geq 2$ and $\mathfrak{p} \in \mathcal{P}$,
 	\[\sum_{N(\mathfrak{p}) \geq x,\, g(\mathfrak{p})=0} \frac{1}{N(\mathfrak{p})} \leq K(x) \quad \text{and} \quad \sum_{N(\mathfrak{p}) \leq x,\, g(\mathfrak{p})=-1} \frac{1}{N(\mathfrak{p})} \geq L(x). \]
 Then one has,
 	\begin{equation*}
 	\begin{split}
 	&|  \{\mathfrak{m}   \in \mathcal{I} | N(\mathfrak{m})\leq x : g(\mathfrak{m})=1 \} | \\ =& (1+ o(1)) | \{ \mathfrak{m} \in \mathcal{I} | N(\mathfrak{m})\leq x: g(\mathfrak{m})=-1 \}| \\ =& (\frac{1}{2}+o(1)) c_{F} x \prod_{\mathfrak{p} \in \mathcal{P}} (1-\frac{1}{N(\mathfrak{p})}) (1+\frac{| g(\mathfrak{p}) |}{N(\mathfrak{p})}+\frac{{| g(\mathfrak{p}^{2}) |}}{N(\mathfrak{p})^2}+ \cdots),
 	\end{split}
 	\end{equation*}
 where $c_{F}$ is residue of $\zeta_F$ at the pole $s=1$.
 \end{lem}
 To prove \lemref{3.1} we need two other lemmas about the mean value of the function $g$, which is defined by
 \[ M(g):= \lim\limits_{x \rightarrow \infty} \frac{1}{N(x;\mathcal{O}_{F})} \sum_{N(\mathfrak{m})\leq x} g(\mathfrak{m}), \]
 where $N(x; \mathcal{O}_{F})$ is the cardinality of integral ideals with norm less than or equals to $x$. Those two lemmas (i.e. \lemref{3.2.}, \lemref{3.3.}) are immediate consequences of certain results proved in \cite{mean} on the topic of arithmetic semigroups (an example is the set of integral ideals $\mathcal{I}$ ). For the convenienence of reading let us rewrite those results with taking the semigroup to be $\mathcal{I}$ (see \cite[Corollary 4.4, Theorem 6.3]{mean}). These results would be used further.
 
 \begin{thm}\label{theorem 6.3}
 Let $h: \mathcal{I} \rightarrow \mathbb{R}$ be a multiplicative function bounded by $1$. Then the mean value $M(h)$ exists. In particular if the series 
 \[\sum_{\mathfrak{p} \in \mathcal{P}} \frac{|1-h(\mathfrak{p})|}{N(\mathfrak{p})}\]
 diverges, then $M(h)$ vanishes to zero.
 \end{thm}
 
 \begin{thm}\label{theorem 4.4}
  Let $h: \mathcal{I} \rightarrow \mathbb{C}$ be a multiplicative function bounded by $1$. Further assume that
  \[\sum_{\mathfrak{p} \in \mathcal{P}} \frac{|1-h(\mathfrak{p})|}{N(\mathfrak{p})} < \infty.\]
  Then $M(h)$ exists and
  \[M(h)= \prod_{\mathfrak{p} \in \mathcal{P}} (1-\frac{1}{N(\mathfrak{p})}) (1+\frac{| h(\mathfrak{p}) |}{N(\mathfrak{p})}+\frac{{| h(\mathfrak{p}^{2}) |}}{N(\mathfrak{p})^2}+ \cdots).\]
 \end{thm}
\thmref{theorem 6.3} and \thmref{theorem 4.4} follows from \cite[Theorem 6.3]{mean}, \cite[Corollary 4.4]{mean} respectively.
 \begin{lem} \label{3.2.}
When $g$ is as given in \lemref{3.1}, $M(g)$ exists and is equal to zero.
 \end{lem}
 \begin{proof}
From the hypothesis of \lemref{3.1}, we note that $\sum_{\mathfrak{p} \in \mathcal{P}} \frac{1-g(\mathfrak{p})}{N(\mathfrak{p})} $ diverges. It follows from the fact that
\[\sum_{\mathfrak{p} \in \mathcal{P}} \frac{1-g(\mathfrak{p})}{N(\mathfrak{p})} \geq \sum_{\mathfrak{p} \in \mathcal{P},g(\mathfrak{p})=-1} \frac{2}{N(\mathfrak{p})}. \]
 Now the lemma immediately follows directly from \thmref{theorem 6.3}.
 \end{proof}
 The next lemma concerns about the mean value of a non-negative multiplicative function on $\mathcal{I}$. It is an obvious consequence of \thmref{theorem 4.4}. We omit the proof.
 \begin{lem} \label{3.3.}
 	Let $l :\mathcal{I} \rightarrow [0,1]$ be a multiplicative function such that $\sum_{\mathfrak{p} \in \mathcal{P}} \frac{1-l(\mathfrak{p})}{N(\mathfrak{p})}$ converges. Then $M(l)= \prod_{\mathfrak{p} \in \mathcal{P}} (1-\frac{1}{N(\mathfrak{p})}) (1+\frac{| l(\mathfrak{p}) |}{N(\mathfrak{p})}+\frac{{| l(\mathfrak{p}^{2}) |}}{N(\mathfrak{p})^2}+ \cdots)$ exists.
 \end{lem}
\noindent\textbf{Proof of \lemref{3.1}.} Recall that $N(x; \mathcal{O}_{F}) \sim c_{F}x$ (see \cite[Theorem 11.1.5]{mur}). From \lemref{3.2.} we get
 \begin{equation} \label{sumg}
  \sum_{N(\mathfrak{m}) \leq x} g(\mathfrak{m})= o(N(x;\mathcal{O}_{F}))= o(x).
\end{equation}
Note that from the hypothesis of \lemref{3.1}, $|g|$ satisfies the conditions of \lemref{3.3.}. It follows from the fact that
\[\sum_{\mathfrak{p} \in \mathcal{P}} \frac{1-|g(\mathfrak{p})|}{N(\mathfrak{p})} = \sum_{\mathfrak{p} \in \mathcal{P},g(\mathfrak{p})=0} \frac{1}{N(\mathfrak{p})}.\]
So putting $l(\mathfrak{m})= | g(\mathfrak{m}) |$ we get from \lemref{3.3.} that
 \begin{equation} \label{summodg}
 \sum_{N(\mathfrak{m}) \leq x} | g(\mathfrak{m}) | = (1+ o(1)) c_{F}x \prod_{\mathfrak{p} \in \mathcal{P}} (1-\frac{1}{N(\mathfrak{p})}) (1+\frac{| g(\mathfrak{p}) |}{N(\mathfrak{p})}+\frac{{| g(\mathfrak{p}^{2}) |}}{N(\mathfrak{p})^2}+ \cdots).
 \end{equation}
Now since $g$ takes only three values $\{-1,0,1\}$, one has
\[\sum_{N(\mathfrak{m}) \leq x} g(\mathfrak{m})= \underset{N(\mathfrak{m}) \leq x}{\sum_{\mathfrak{m} \in \mathcal{I}, g(\mathfrak{m})=1}} 1 - \underset{N(\mathfrak{m}) \leq x}{\sum_{\mathfrak{m} \in \mathcal{I}, g(\mathfrak{m})=-1}} 1 \]
and
\[ \sum_{N(\mathfrak{m}) \leq x} |g(\mathfrak{m})|= \underset{N(\mathfrak{m}) \leq x}{\sum_{\mathfrak{m} \in \mathcal{I}, g(\mathfrak{m})=1}} 1 + \underset{N(\mathfrak{m}) \leq x}{\sum_{\mathfrak{m} \in \mathcal{I}, g(\mathfrak{m})=-1}} 1.\]
 Hence the \lemref{3.1} follows from \eqref{sumg} and \eqref{summodg}. \qed
 
\lemref{3.1} leads us to the following result.
\begin{lem} \label{3.4}
	Let $K,L : \mathbb{R}^{+} \rightarrow \mathbb{R}^{+} $ be such that $K(x) \rightarrow 0$ and $L(x) \rightarrow \infty$ for $x \rightarrow \infty$. Let $j : \mathcal{I} \rightarrow \mathbb{R}$ be a multiplicative function such that for every $x \geq 2$ and $\mathfrak{p} \in \mathcal{P}$,
\[	\sum_{N(\mathfrak{p}) \geq x,\, j(\mathfrak{p})=0} \frac{1}{N(\mathfrak{p})} \leq K(x) \quad \text{and} \quad \sum_{N(\mathfrak{p}) \leq x,\, j(\mathfrak{p})<0} \frac{1}{N(\mathfrak{p})} \geq L(x). \]
Then one has
\begin{equation*}
\begin{split}
	&| \{\mathfrak{m} \in \mathcal{I} | N(\mathfrak{m})\leq x : j(\mathfrak{m}) > 0 \} | \\ =& (1+ o(1)) | \{ \mathfrak{m} \in \mathcal{I} | N(\mathfrak{m})\leq x: j(\mathfrak{m}) < 0 \} | \\ =& (\frac{1}{2}+o(1)) c_{F} x \prod_{\mathfrak{p} \in \mathcal{P}} (1-\frac{1}{N(\mathfrak{p})}) (1+\frac{ s(\mathfrak{p}) }{N(\mathfrak{p})}+\frac{{ s(\mathfrak{p}^{2}) }}{N(\mathfrak{p})^2}+ \cdots),
\end{split}
\end{equation*}
	where $s(\mathfrak{m})$ is $0$ or $1$ according as $j(\mathfrak{m})=0$ or $j(\mathfrak{m}) \neq 0$.		
\end{lem}
\begin{proof}
This is obvious by applying \lemref{3.1} to the multiplicative function which takes value $0$ when $j$ is $0$ and takes value $\frac{j(\mathfrak{m})}{| j(\mathfrak{m}) |}$ otherwise.	
\end{proof}
 We would need the following lemma to apply \lemref{3.4} to prove \thmref{ratio}.
 \begin{lem} \label{3.5}
Let $f$ be a primitive adelic Hilbert cusp form of weight $(k_1, k_2,...,k_n) \in (2 \mathbb{Z})^{n}$ and full level. Let ${C(\mathfrak{m})}$ denote its Fourier coefficients. Then one has
\[\sum_{N(\mathfrak{p})\geq x, C(\mathfrak{p})=0} \frac{1}{N(\mathfrak{p})} = o(1)\quad \text{and} \quad \sum_{C(\mathfrak{p}) <0} \frac{1}{N(\mathfrak{p})}\]
diverges.
 \end{lem}
 \begin{proof}
 The first result follows immediately from \cite[pp 162-163]{serre}. For the second case, from the Sato-Tate theorem one has $C(\mathfrak{p}) < 0$ for a positive proportion of prime ideals in $\mathcal{P}$, when they are counted according to their norms (see \cite[pages 1-6]{sato}). So let for some large enough $x$, $E(x)$ be the set of prime ideals defined by 
 \[E(x):= \{\mathfrak{p} | N(\mathfrak{p}) \leq x, C(\mathfrak{p}) <0\}.\]
 By Sato-Tate theorem, there exists an $\delta >0$, such that
 \[|E(x)| > \delta \cdot \frac{x}{\log x}. \]
 So by partial summation formula one has
 \[\sum_{\mathfrak{p} \in E(x)} \frac{1}{N(\mathfrak{p})} >  \frac{\delta x}{x \log x}+\int_{2}^{x} \frac{\delta x}{x^2 \log x} dx=  \frac{\delta}{ \log x}+\int_{2}^{x}  \frac{\delta}{x \log x} dx.\]
 We note that the integral is divergent and hence the lemma follows.
 
\end{proof}
\noindent\textbf{Proof of \thmref{ratio}.}
\lemref{3.5} implies that we can put $j(\mathfrak{m})=C(\mathfrak{m})$ in \lemref{3.4} and hence the theorem follows. \qed


\begin{thebibliography}{22}

\bibitem{sato} T. Barnet-Lamb, T. Gee, D.Geraghty:  {\em The Sato-Tate conjecture for Hilbert modular forms}, J. Amer. Math. Soc. 24 (2011), 411--469.

\bibitem{bla} D. Blasius: {\em Hilbert Modular forms and the Ramanujan conjecture}, Noncommutative geometry and Number Theory (2006), 35--56.

\bibitem{ivic} A. Ivi\'c: {\em On squarefree numbers with restricted prime factors.}, Studia. Sci. Math. Hungar 20 (1985), 189--192.

\bibitem{iwaniec}  H. Iwaniec, E. Kowalski : {\em Analytic Number Theory}, Volume 53, American mathemaical society, Colloquium publications.

\bibitem{sound} E. Kowalski, Y.K. Lau, K. Soundararajan, J. Wu: {\em On modular signs}, Math. Proc. Camb. Phil. Soc. 149 (2010), 389--411

\bibitem{lebacque} P. Lebacque: {\em Generalized Merten's and Brauer-Siegel Theorems}, Acta Arith. 130 (2007), 333--350.

\bibitem{mean} L. G. Lucht, K. Reifenrath: {\em Mean value theorems in arithmetic semigroups }, Acta Math. Hungar. 93 (1-2) (2001), 27--57.

\bibitem{km1} K. Matom\"aki : {\em On signs of Fourier coefficients of cusp forms}, Math. Proc. Camb. Phil. Soc., 152 (2012),  207--222.

\bibitem{km} K. Matom\"aki, M.Radziwi\l{}\l{}: {\em Sign changes of Hecke eigenvalues},  Geometric and Functional analysis (2015), 1937--1955. 

\bibitem{jm} J. Meher, N. Tanabe: {\em Sign changes of Fourier coefficients of Hilbert modular forms},  J. number theory. 145 (2014), 230--244.

\bibitem{moore} P. Moree: {\em An interval result for the number field $\psi(x,y)$ function}, Manuscripta. math. 76 (1992) 437--450.

\bibitem{mur} M. R. Murty: {\em Problems in Algebraic number theory}, second edition, springer.

\bibitem{qu} Y. Qu: {\em Linnik type problems for automorphic $L-$ functons}, J. Number Theory, 130 (2010), 786--802.

\bibitem{raghuram} A. Raghuram, N. Tanabe: {\em Notes on the arithmetic of Hilbert modular forms}, J. Ramanujan Math. Soc., 26 (2011), 261--319.

\bibitem{serre} J. P. Serre: {\em Quelques applications du theoreme de densite de Chebotarev}, Publications mathematiques de l'I.H.E.S., tome 54 (1981), 123--201.

\bibitem{shimu} G. Shimura: {\em The special values of the zeta functions associated with Hilbert modular forms}, Duke Math. J. 45 (1978), 637--679 .

\bibitem{Tenenbaum} G. Tenenbaum: {\em Introduction to analytic and probabilistic number theory}, Cambridge studies in advanced mathematics 46, (Cambridge university press, 1995).

\end{thebibliography}
\end{document}